\documentclass[10pt]{amsart}
\usepackage{amscd}

\usepackage{amssymb}

\newcommand{\bT}{{\mathbb T}}

\newcommand{\bS}{{\mathbb S}} 

\newcommand{\bZ}{{\mathbb Z}}
\newcommand{\bR}{{\mathbb R}}

\newcommand{\bC}{{\mathbb C}}

\newcommand{\bF}{{\mathbb F}}

\newcommand{\bG}{{\mathbb G}}

\newtheorem{thm}{Theorem}[section]
\newtheorem{lemma}[thm]{Lemma}
\newtheorem{cor}[thm]{Corollary}

\numberwithin{equation}{section}

\begin{document}

\title[Chow rings of flag varieties]{computing Chow rings 
of twisted flag varieties for subgroups
of an algebraic group}
 
\author{Nobuaki Yagita}

\address{ faculty of Education, 
Ibaraki University,
Mito, Ibaraki, Japan}
 
\email{ nobuaki.yagita.math@vc.ibaraki.ac.jp, }


\begin{abstract}
Let $G_k$ be a split algebraic group over a field $k$.  
Given a non-trivial $G_k$-torsor $\bG$, we consider the (twisted)
flag variety $\bF=\bG/B_k$ for the Borel subgroup $B_k$ containing in $G_k$.  
 The purpose of this paper is the study how to compute
the Chow rings $CH^*(\bF)$ for various $\bF$ over some 
extension fields $K$ of $k$. 

\end{abstract}

\maketitle

\section{Introduction}

Let $G_k$ be a split algebraic group over a field $k$.  
Given a non-trivial $G_k$-torsor $\bG$, we consider the (twisted)
flag variety $\bF=\bG/B_k$ for the Borel subgroup $B_k$ containing in $G_k$.  
 The purpose of this paper is the study how to compute
the Chow rings $CH^*(\bF)$ for various $\bF$ over some 
extension fields $K$ of $k$. 
Non-split Chow rings $CH^*(\bF)$ are studied mainly when $\bG$ are versal (which is thought the  most twisted cases).  
We will study  non-versal cases also.

Let  $G'\subset G$
( with some weak condition).  Let us write 
by  $\bF(v)$  the flag variety for the $versal$ torsor  $\bG(v)$ and $K$  is the field such that $\bG(v)$ is defined over $K$,
(notations $\bF'(v)$, $K'$ are defined similarly for $G'$).
That is, $K=k(GL_N/G_k)$ for embedding $G_k\subset GL_N$ for a large $N$.  Since $G'\subset G$, we see $K\subset K'$.
Then we note the isomorphism of $K'$-motives
\[ \bF(v)|_{K'}\cong \bF'(v)\otimes  \bT(G',G))\qquad(Theorem \ 4.2)\]
for a sum $\bT(G',G)$ of Tate motives. 
We will try to compute the versal $CH^*(\bF(v))$ from
the sequences of some standard inclusions of
simple simply connected groups  
\[ G=G_n\supset G_{n-1}\supset ...\supset G_1.\]
\[ CH^*(\bF^n(v)) \to CH^*(\bF^{n-1}(v))\to ...\to 
CH^*(\bF^1(v)),\]
where $\bF^i(v)$ is the versal flag for the torsor $\bG_i(v)$.

In this paper, we will mainly discuss the Chow ring of the generalized Rost motive $R(\bG))$ instead of $CH^*(\bF)$
such that $CH^*(\bF)\cong CH^*(R(\bG))\otimes T$ for some sum $T$ of Tate motives.  

For example, when $G=Spin_{11}$ and 
$G'=Spin_{7}$, we see that
\[ CH^*(R(\bG(v))|_{K'})/2\cong CH^*(R(\bG'(v)))/2\otimes \Lambda(y_{10})\]
where $2deg(y_{10})=|y_{10}|=10.$
(Hence $CH^*(\bT(G',G))\cong \Lambda(y_{10})$.)
Here we have seen 
\[ CH^*(R(\bG(v)))/2\cong \bZ/2
\{1,c_2,c_3,c_4,c_5,c_2c_4,e_{8}\},
\quad |c_i|=2i,\ |e_8|=16,\] \[
\ \ CH^*(R(\bG'(v)))/2\cong \bZ/2\{1,c_2,c_3\}\]
where $ \bZ/2\{a,...,b\}$ is the $\bZ/2$-free module generated by $a,...,b$. 
The restriction to 
$ CH^*(R(\bG(v))|_{K'})/2$ is given as 
$ c_4,c_5\mapsto 0,\ \
e_8\mapsto c_3y_{10}.$

\section{versal torsor and specialization}

Let $k$ be a field with $ch(k)=0$.
Let $G$ be a connected compact Lie group and $G_k$ be the corresponding split algebraic group over $k$ corresponding to the Lie group $G$.

The versal $G_k$-torsor $\bG(v)$ is defined as follows
([Ka1.Ka2]).
Let $G_k\subset GL_N$ be an embedding for some large $N$.
Let $S=GL_N/G_k$ and $k(S)$ be its function field. The versal
$G_{k(S)}$-torsor $\bG(v)$ is defined as a generic fiber $E$ of the projection $GL_N\to S$ induced from the map $Spec(k(S))\to S$.

We fix the field $k$.
Let $K$ be an extension of $k$. Let us write by  $\bG(K)$ a $G_K$-torsor and the flag variety $\bF(K)=\bG(K)/B_K$ over $K$. (Of course $\bF(K)$ is not uniquely determined.) Then we have the natural map (specialization)
\[ sp_{CH}: CH^*(\bF(v))\to CH^*( \bF(K))\]
where $\bF(v)=\bG(v)/B_k$ is the versal flag variety.
Therefore the versal flag variety $\bF$ is thought as the most 
twisted (complete) flag variety over $K$ in those defined over
extensions $K$ over $k$.

\section{Generalized Rost motives}

Recall that it is known by Borel and Toda ([Bo], [To], [Ya3])
\[CH^*(G_k/B_k)\cong H^{2*}(G/T)\cong P(y)\otimes S(t)/(b)\]
\[ where  \quad 
 P(y)=\bZ_{(p)}[y_1,...,y_s]/(y_1^{p^{r_1}},....,y_s^{p^{r_s}})\quad |y_i|=even,\]
\[ S(t)/(b)=\bZ_{(p)}[t_1,...,t_{\ell}]/(b_1,...,b_{\ell}),\quad |t_i|=2\]
for some regular sequence $(b_1,...,b_{\ell})$ in $ S(t)=\bZ_{(p)}[t_1,...,t_{\ell}]$.  (We also use the additive identify
\[ P(y)\cong grP(y)= \Lambda(y_1,y_1^p,...,y_1^{p^{r_1-1}},...,
y_s,y_s^p,...,y_s^{p^{r_s-1}})\]
for the notation that  $\Lambda(a,...,b)=\bZ_{(p)}[a,...,b]/(a^p,...,b^p)$.

 Let $\bG(K)$ be a $G_K$-torsor and $\bF(K)=\bG(K)/B_K$.
Let $J(\bG(K))=(j_1,...,j_s)$ be the J-invariant so that $y_i^{p^{j_i}}
\in Im(res_{CH}(\bF(K))$ for the restriction map $res_{CH}: CH^*(\bF(K))\to CH^*(\bF(K)|_{\bar K})$.  Let us write
\[P(y^J)=\bZ_{(p)}[y_1^{p^{j_1}},...,y_s^{p^{j_s}}]/(y_1^{p^{r_1}},....,y_s^{p^{r_s}}).\]
Then the main theorem  of Petrov-Semenov-Zainoulline is written as

\begin{thm} ([Pe-Se-Za])
Given a $G_K$-torsor $\bG(K)$,
there is an irreducible motive $R(\bG(K))$ over $K$ such that 
the motive $M(\bF(K))$ is decomposed as
\[ M(\bF(K))\cong R(\bG(K))\otimes T(K)\]
where $T(K)$ is a sum of Tate motives such that for $J=J(\bG(K))$
\[ CH^*(T(K))\cong P(y^{J})\otimes S(t)/(b).\]
\end{thm}

{\bf Remark.} Each $\bG(K)/B_k)$ is $generically$ $split$
(i.e., split over  the function field 
$K(\bG(K)/B_k)$, see 4.1 in [Pe-Se-Za]).

{\bf Remark.}  The motive $R(\bG(K))$ are called generalized 
Rost motive.  Some cases they are original Rost motive $R_n$
$n\le 2$, but in general $R_n\not \cong R(\bG(K))$ [Ro],
[Pe-Se-Za]).

From the above theorem, we have the (additive) isomorphisms of motives over the algebraic closure $\bar K$
(  for  $\bar \bF(K)=\bF(K)\otimes \bar K)$
\[ P(y)\otimes S(t)/(b)\cong  \bar \bF(K)\cong 
CH^*(\bar R(\bG(K)))\otimes P(y^J)\otimes S(t)/(b).\]
\begin{cor}
We have   $ P(y)/(P(y^J) \cong  CH^*(\bar R(\bG(K))).$
\end{cor}

There is a specialization $CH^*(R(\bG)(v))\to CH^*(R(\bG(K))$,
and we have
\[ J(\bG(v)|_K)\ge J(\bG(K)) ,\quad i.e., P(y^{J(\bG(v)|_K)})\subset P(y^{J(\bG(K))}).\]
In particular, the versal case, $J(\bG(v))=(p^{r_1},...,p^{r_s})$ and  $P(y^J)=\bZ_{(p)}$.
\begin{cor} Suppose   $J=J(\bG(v)|_K)= J(\bG(K)).$  Then we have the isomorphism
\[  CH^*(R(\bG(v))|_K)\cong CH^*(R(\bG(K))\otimes P(y^J),\]
\[  that \ is, \quad  CH^*(R(\bG(v))|_K)/(P(y^J))\cong CH^*(R(\bG(K)).\]
\end{cor}
\begin{proof}
We consider the
restriction to the algebraic closure $\bar K$.
First note that 
\[ CH^*((R(\bG(v))|_K)|_{\bar K}))\cong CH^*(\bar R(\bG))\cong P(y).\]
 On the other hand (aditively)
\[CH^*(\bar R(\bG(K))\otimes P(y^J)\cong P(y)/P(y^J)\otimes P(y^J)\cong P(y).\]
We get the isomorphism in this corollary, from a theorem by Vishik and Zinoulline (Theorem 2.1 in [Vi-Za]) that the isomorphism of the Chow ring of motives over an algebraically closed field $\bar K$ induces that in the field $K$ itself (when they are generically split).
\end{proof}

\section{subgroups $G'\subset G$}

Let $G'$ be a subgroup of $G$.
Let us write $K=k(GL_N/G_k)$ (resp. $K'=k(GL_N/G'))$ be the 
field such that $\bG$ (resp. $\bG'$)
 is the versal $G_k$-torsor.  Then the inclusion $G'\subset G$ induces a map $K\to K'$, so that $K'$ is a filed  extension of $K$.
The inclusion $G'\subset G$ induces
\begin{lemma}   We have the natural  map
 $R(\bG(v))|_{K'} \to R(\bG'(v))$
of motives over $K'$.
\end{lemma}
\begin{proof}
We consider the following diagram
\[\begin{CD}
\bG'(v) @>>>{\to} @>>>{\to} @>>> GL_N \\
@VVV  @. @. @VV{=}V  \\
{ } @.  {\bG(v)|_{K'}} @>>>{\bG(v)} @>>> GL_N \\
@VVV @VVV @VVV@VVV  \\
Spec(K') @>>{=}> Spec(K') @>>> {\to} @>>> GL_N/G_k'\\
@. @V{=}VV @VVV @VVV \\
@. Spec(K') @>>> Spec(K) @>>> GL_N/G_k
\end{CD} \]

First note that $\bG(v)|_{K'}$ is a pullback of $GL_N$ by the map
$Spec(K')\to Spec(K)\to  GL_N/G_k$.  In particular
\[Im(\bG(v)|_{K'} \to GL_N \to GL_N/G_k) = 
Im(\bG(v)|_{K'} \to Spec(K') \to Spec(K)\to GL_N/G_k).\]

On the other hand $\bG'(v)$ is a pullback of $GL_N$ by the map
$Spec(K') \to GL_N/G_k'$.   In particular
\[Im(\bG'(v) \to GL_N \to GL_N/G_k')= 
Im(\bG'(v) \to Spec(K')\to GL_N/G_k').\]
Of course their images  in $GL_N/G_k$ are the same.
Hence by the universality of the pullback for $\bG(v)|_{K'}$,
we have a map $\bG'(v)\to \bG(v)|_{K'}$.

\end{proof}

 Recall that $P_{G}(y)$ and $P_{G'}(y)$ are the polynomials
in $H^*(G/T)$ and $
H^*(G'/T)$ respectively (hence polynomials in $CH^*(G_k/B_k)$ and 
$CH^*(G_k'/B_k')$).

\begin{thm} Suppose $i_P^*:P_G(y)\to P_{G'}(y)$ is surjective.
Then as motives
\[ R(\bG(v)|_{K'})\cong R(\bG'(v)),\]
\[ 
 R(\bG(v))|_{K'}\cong R(\bG'(v))\otimes Ker(i_{P}^*).\]
\end{thm}
\begin{proof}
We consider maps of motives over $K$
\[ \begin{CD} R(\bG(v)) @>(1)>> R(\bG(v))|_{K'} @>(2)>> \bar R(\bG(v)) \\
 @.     @V(3)VV   @V(4)VV   \\
   @.   R(\bG'(v)) @>(5)>>   \bar R(\bG'(v)).
\end{CD}\]
From the theorem by Petrov-Semenov-Zainoulline, we have
\[ R(\bG)(v)|_{K'}\cong R(\bG(v)|_{K'})\otimes T',\quad \]
where $T'$ is a sum of $K'$-Tate motives.  
Since $R(\bG'(v))$ is non Tate irreducible motive, the image
$(3)(T')=0$ in $R(\bG'(v))$.

 Recall $(2)(1)R(\bG(v))\cong P_G(y)$, and hence 
it isomorphic also to $(2)R(\bG(v))|_{K'}$.
Hence
$(4)(2)R(\bG(v))|_{K'}=P_{G'}.$ and $(5)(3)R(\bG(v))|_{K'}$ is a non zero  map.
Moreover $(3)R(\bG|_{K'})$ is non zero since $(3)T'=0$.

Since both $R(\bG(v)|_{K'})$ and $ R(\bG'(v))$ are irreducible
and is isomorphic to $P_{G'}(y)$ in $\bar R(\bG'(v))$, we get
\[ R(\bG(v)|_{K'})\cong R(\bG'(v)),\quad and \quad T'\cong P_G(y)/P_{G'}(y).\]
\end{proof}

More generally, we have
\begin{cor} Suppose that there is an extension $K''$ of $K$
and  an irreducible motive $R''$ over $K''$  such that there is a map 
of $K''$-motives 
$ R(\bG)|_{K''} \to R''$
and that $i_P^*:\bar R(\bG) \to \bar R''$ is surjective.
Then as motives
\[ R(\bG)|_{K''}\cong R''\otimes Ker(i_{P}^*).\]
\end{cor}
\begin{proof}
The arguments in the preceding theorem work, exchanging
$K'$ to $K''$, and $R(\bG')$ to $R''$.
\end{proof}

\section{Subgroups of the exceptional group $E_8$, $p=2$}

The exceptional Lie group $E_8$ has a $p$-torsion in its cohomology when $p=2,3,5$ [Mi-Tod].
For $G=E_8, p=2$, we see  
\[ CH^*(\bar R(\bG))/2\cong P(y)/2\cong \bZ/2[y_1,y_2,y_3,y_4]/(y_1^8,y_2^4,y_3^2,y_4^2)\]
\[ \stackrel{add.}{\cong} \Lambda(y_1,y_1^2,y_1^4,y_2,y_2^2,y_3,y_4)\quad |y_1|=6,|y_2|=10,|y_3|=18,|y_4|=30.\]
For example we take
subgroups $G'=E_7,Spin_{11},Spin_{7},Spin_{5}$, or the (original) Rost motive
$R_4$ as  $R(G'(v))$.


{\bf 5.1.}  We consider sequence of subgroups
\[ (1)\quad E_8\stackrel{y_1^2,y_1^4,y_2^2}{\gets}\begin{cases}
   \stackrel{y_{4}}{\gets}E_7 
\stackrel{y_{3}}{\gets}Spin_{11} 
\stackrel{y_{2}}{\gets}Spin_{7} 
\stackrel{y_{1}}{\gets}pt.\\  \\
 \stackrel{y_{1},y_{2}.y_{3}}{\gets}R_4\ (Rost\ motive)
\stackrel{y_4}{\gets} pt.
\end{cases} \]  
Here for  subgroups $G'\supset G''$, the graph 
$G' \stackrel{y_a}{\gets}... \stackrel{y_b}{\gets}G''$ means that
\[ P_{G'}(y) \cong  \Lambda(y_a,...,y_b)\otimes P_{G''}(y).\]
For example, we see that $Spin_{11}\stackrel{y_2}{\gets}
Spin_7 \stackrel{y_1}{\gets} pt.$ implies 
\[ P_{Spin_{11}}(y)\cong \Lambda(y_2)\otimes P_{Spin_{7}}(y)\cong \Lambda(y_2,y_1), \quad P_{Spin_7}(y)\cong \Lambda(y_1).\]
Hence  the maps $i_P^*$ in Theorem 4.2  are surjective.
Thus we have 
\begin{lemma}  
In (1), for  subgroups $G'\supset G''$, the graph 
$G' \stackrel{y_a}{\gets}... \stackrel{y_b}{\gets}G''$ means that
\[CH^*(R(\bG'(v))|_{K_{G''}}) \cong  \Lambda(y_a,...,y_b)\otimes CH^*(R(\bG''(v)))\]
where $K_{G''}=K(GL_N/G'')$ is the  field such that $\bG''$
 becomes versal, and hence 
 $\Lambda(y_a,...,y_b)\cong J(R(\bG')|_{K_{G''}}).$
\end{lemma}
For example
\[CH^*(R(Spin_{11}(v))|_{K_{Spin_7}}) \cong  \Lambda(y_2)\otimes CH^*(R(Spin_{7}(v))),\quad \]
In fact, we can compute ($|c_i|=2i, |e_8|=16$)
\[CH^*(R(Spin_{11}(v))\cong \bZ_{(2)} \{1,c_3,c_5,e_8\}
\oplus \bZ/2\{c_2,c_4,c_2c_4\},\]
\[CH^*(R(Spin_{7}(v)))\cong  \bZ_{(2)}\{1,c_3\}\oplus \bZ/2\{c_2\}.\]
We have $c_4\mapsto 0$ and $c_5\mapsto 2y_2$, $e_8\mapsto c_3y_2$  in  $CH^*(R(Spin_{11}(v))|_{K_{Spin_7}}$.

Each $CH^*(R(\bG'(v)))$ should be computed from the right hand side to the left hand sided but it seems difficult in general.
To treat elements $c_2,c_4$ above, we use the algebraic cobordism.

{\bf 5.2.} 
Here we recall the algebraic cobordism $\Omega^*(X)$ such that
\[ CH^*(X)\cong \Omega^*(X)\otimes_{\Omega^*}\bZ_{(p)},\quad where \ \
\Omega^*\cong \bZ_{(p)}[v_1,v_2,...]\ \ with\ |v_i|=-2(p^i-1).
\]  For the restriction map 
$ res_{\Omega}: \Omega^*(X)\to \Omega^*(\bar X).$
let us write 
\[Res_{\Omega}(X)=Im(res_{\Omega})\subset 
gr \Omega^*(\bar X)\cong \Omega^*\otimes CH^*(\bar X)
.\]
(For details, see [Ya4].)
We consider the another  sequences representing 
$Res_{\Omega}(R(\bG'(v))$.
 For example
$(2,v_1^3)y_3$ (in the following graph)  is the invariant ideal generated by $2y_3,v_1^3y_3$ in
$\Omega^*(\bar R(\bG'(v))$ for $G'=E_7$.

\[ (2)\quad E_8\stackrel{...}{\gets}\begin{cases}
\stackrel{2y_{4},...}{\gets}E_7 
\stackrel{(2,v_1^3)y_{3},...}{\gets}Spin_{11} 
\stackrel{((2,v_1)y_{2}, (2,v_1^2)y_1y_2)}{\gets}Spin_{7} 
\stackrel{(v_1,2)y_{1}}{\gets}pt.
 \\  \\
 \stackrel{y_{1},y_{2}.y_{3}}{\gets}R_4\ (Rost\ motive)
\stackrel{(2,v_1,v_2,v_3)y_4}{\gets}pt.
\end{cases} \] 
\begin{lemma}
Given 
a subgroup $G''$ of $G'$, the graph 
$G' \stackrel{J_ay_a}{\gets}... \stackrel{J_by_b}{\gets}G''$ means  that
there is a (natural) map
\[Res_{\Omega}(R(\bG'(v))) \to   J_ay_a+...+J_by_b+
Res_{\Omega}(R(\bG''(v))).\]
\end{lemma}  For example,  we can compute
\[Res_{\Omega}(R(Spin_{11}(v))) \
\cong (2,v_1)y_2+(2,v_1^2)y_1y_2+Res_{\Omega}(R(Spin_7(v))).\]
\[Res_{\Omega}(R(Spin_7(v)))\cong (2,v_1)y_1\oplus \Omega^*.\]
Here the restrictions are given as 
\[c_2\mapsto v_1y_1, c_3\mapsto 2y_2,
c_4\mapsto v_1y_3,c_5\mapsto 2y_2,e_8\mapsto 2y_1y_2.\]
Hence, in these cases, the  map $\to$ in the above lemma are isomorphic.

However, in general, the above map is not surjective nor injective. For example, let $G'=E_7, G''=Spin_{11}$ the element $v_1y_2$ does  not come from
that for $E_7$ where $(2,v_1^3)y_2$ only exists.
 
{\bf 5.3.}    At last, we consider Chow groups.
We use the notations in \cite{YaC} for $b_i,c_j\in CH^*(R(\bG))$.
We also use $a_0,..,a_3$ for $2y_4,..,v_3y_4$ respectively. 
Consider the another diagram for Chow rings

\[ (3)\quad E_8\stackrel{...}{\gets}\begin{cases}
   \stackrel{a_0,...}{\gets}E_7
\stackrel{b_4,b_1b_3,...}{\gets}Spin_{11} 
\stackrel{c_5,c_4, c_2c_4,e_8}{\gets}Spin_{7} 
\stackrel{c_3,c_2}{\gets}pt. \\  \\
\stackrel{y_{1},y_{2}.y_{3}}{\gets}R_4\ (Rost\ motive)
\stackrel{a_0,...,a_3}{\gets}pt.
\end{cases} \] 

\begin{lemma}
For a subgroup $G''$ of $G'$, the graph 
$G' \stackrel{d_a}{\gets}... \stackrel{d_b}{\gets}G''$ means that
there is a natural map
\[ CH^* (R(\bG'(v)))/2 \to  \bZ/2\{d_a,...d_b\}
+  CH^*(R(\bG''(v))/2.\]
\end{lemma}
For example, we had seen 
\[  CH^* (R(Spin_{11}(v)))/2\cong 
\bZ/2\{c_4,c_5,c_2c_4,c_8\}+ \tilde CH^*(R(Spin_{7}(v))/2.\]

Here we explain the other elements in $CH^*(\bG'(v))$. 
When $G=E_7, G'=Spin_{11}$,  elements in $CH^*(\bG)$ are denoted by $b_i$.
For example, $b_1=c_2,b_2=c_3,b_3=c_5, $.  But $c_4=v_1y_2$ in 
$CH^*(\bG')$ does not come from  $CH^*(\bG)$. 
Moreover $c_2c_3=b_5\not =0$ in $CH^*(\bG)$.
We write down some elements in $CH^*(\bG((v))$ as 
\[ (b_4,b_1b_3)\mapsto (2,v_1^2)y_3, \ (b_6,b_1b_5)\mapsto (2,v_1^2)y_1y_3,\ (b_7,b_1b_4)\mapsto (2,v_1^3)y_2y_3.\]

When $G'=R_4$. we define
$ (a_0,a_1,a_2,a_3)\mapsto (2,v_1,v_2,v_3)y_4. $

As for $Kernel(CH^*(\bG)\to CH^*( R(\bG'))$ when $G'=E_7$, we  
know almost nothing.

\section{The exceptional group $E_8$, $p=3$.}
For $G'=E_7,p=3$, we see  
$ CH^*(\bar R(\bG'(v)))\cong P(y)'\cong \bZ/3[y]/(y^3)$, for $|y|=8$.  
For $G=E_8,p=3$, we see  
\[ CH^*(\bar R(\bG(v)))/3\cong P(y)/3\cong \bZ/3[y,y']/(y^3,(y')^3)\quad |y|=8, |y'|=20.\]
We have the diagram
\[ (1)\quad E_8  \stackrel{y', (y')^2, yy',y(y')^2, y^2(y'),(yy')^2}{\gets}E_7 
\stackrel{y, y^2}{\gets}pt.
 \] 

We find elements $b_1,...,b_8$ in $CH^*(\bG(v))$ (note $b_1,b_2,b_3$ in $CH^*(\bG'(v))$  such that
the restriction maps $res_{\Omega}$ are given as 
\[ (b_2,b_1) \to (3,v_1)\{y\},\quad 
   (b_3,b_1^2) \to (3,v_1^2)\{y^2\},\quad 
(b_4,b_1b_2-v_1b_3)\to (3,v_1^2)\{y'\},\]
\[ (b_5, b_1b_3) \to (3,v_1^2)\{yy'\},\quad 
   (b_6, b_1^2b_3) \to (3,v_1^3)\{y^2y'\},\quad 
(b_7, b_3^2-2v_1b_6)\to (3,v_1^2)\{(y')^2\},\]
\[ (b_8,b_1b_6) \to (3,v_1^2)\{y(y')^2\},\qquad
(b_2b_8,b_1b_8,b_1^2b_6)\to (9,3v_1,v_1^3)\{(yy')^2\}.\]
We have the diagram
\[ (2)\quad E_8  \stackrel{(3,v_1^2)y', (3,v_1^2)(y')^2, ...}
{\gets} E_7
\stackrel{(3,v_1)y, (3,v_1^2)y^2}{\gets}pt.
 \]

The Chow group $CH^*(\bG(v))/3$ seems quite complicated 
\[CH^*(\bG(v))/3\supset
\bZ/3\{1,b_1,b_2,...,b_8, \ b_1^2,b_1b_2,b_1b_3,b_1^2b_3,b_3^2,b_1b_6,b_1^2b_6,\ b_1b_8,b_2b_8\}.\]
The above inclusion seems far from an isomorphism.

\section{The exceptional group $E_8$, $p=5$}

Let $G=E_8$ and $p=5$.  Then
$P(y)/5\cong \bZ/5[y]/(y^5)$ with $|y|=12$.  Hence 
$R(\bG(v))$ is the  original  Rost motive $R_2$ and its Chow ring is  well known.  The graph for $P(y)$ and the restriction $res_{\Omega}$ are written   
\[ (1)\ \ E_8 \stackrel{y,y^2,...,y^{4}}{\longleftarrow} pt,\qquad 
(2)\ \ E_8 \stackrel{ (5,v_1)\{y,..., y^4\}}{\longleftarrow} pt.\]

We can take $e_1,...,e_4, f_1,...,f_4$ in $CH^*(R(\bG(v))$ such that
\[ e_i \mapsto 5y^i\quad f_i\mapsto v_1y^i\quad |e_i|=12i,\ |f_i|=12i-8.\]
Then $CH^*(\bG(v))/5\cong \bZ/5\{1,e_1,...,e_4,f_1,...,f_4\}.$
That is 
\[(3)\ \ E_8 \stackrel{ e_1,f_1,e_2,f_2,...,.e_4,f_4}{\longleftarrow} pt.\]

\section{The spin groups $Spin_{n}$, $p=2$}

Now we consider  
$Spin(2\ell+1)$.
Throughout this section, let $p=2$,  $G=Spin(2\ell+1)$.
First note that 
for $G''=Spin(2\ell+2)$, we know $R(\bG''(v))\cong R(\bG(v))$.
Hence we consider only $G=Spin(2\ell+1)$ here.
In this section, the suffixes mean twice of their degree.
(The elements $y_1,y_2,y_3,y_4$ in $\S 5$ are written by
$y_6,y_{10},y_{18},y_{30}$ in this section.)

It is known that 
\[ grP(y)\cong \otimes _{2i\not =2^j}\Lambda(y_{2i})\cong
\Lambda(y_6,y_{10},y_{12},...,y_{2\bar \ell})\]
where $\bar \ell=\ell-1$ if $\ell=2^j$ for some $j$, and
$\bar \ell=\ell$ otherwise.
Hence for $G'=Spin(2\ell'+1)$ the map $P_{G}(y)\to P_{G'}(y)$
in Theorem 4.2 is surjective.
\begin{lemma}  Let us write $\bG(v)$ by $\bS pin_n$ when $G=Spin_n$.  We have
\[ CH^*(R(\bS pin_{2\ell+1}|_{K'})/2\cong
CH^*(R(\bS pin_{2\ell'+1}))/2\otimes \Lambda(y_{2\bar \ell'+2},...,
y_{2\bar \ell}).\]
\end{lemma}

The Chow goup
 $CH^*(R(\bG'))/2$ is not computed yet
(for general $\ell$),
while we have the following lemmas.
\begin{lemma}
Let $2^t\le \ell<2^{t+1}$and $\bG(v)$ be versal. 
Then there  is a surjection
\[ \Lambda(c_2,...c_{\bar \ell})\otimes \bZ/2[e_{2^{t+1}}]
\twoheadrightarrow CH^*(R(\bG(v))) /2, \quad |e_j|=2j.\]
\end{lemma}

\begin{lemma}  Let us write the restriction 
$res': CH^*(R(\bG(v))
 /2\to 
CH^*(R(\bG(v))|_{K'})/2$.
Then we have \[Ker(res')\supset Ideal(c_{\bar \ell'+1},...,c_{\bar \ell}).\]
\end{lemma}

For $\ell\le 10$. we have the diagram
\[(1)\quad Spin_{21}\stackrel{y_{20}}{\gets} 
Spin_{19}\stackrel{y_{18}}{\gets}
Spin_{17}\stackrel{0}{\gets}
Spin_{15}\stackrel{y_{14}}{\gets}
Spin_{13}\stackrel{y_{12}}{\gets} Spin_{11}.\]
Here we note 
\[R(\bS pin_{17})|_{K(\bS pin_{15})}
\cong R(\bS pin_{15}) \quad  (but  \quad R(\bS pin_{17})\not 
\cong 
R(\bS pin_{15})).  \]

Next, we consider elements in $\Omega^*\otimes CH^*(\bar X)$
\[(2)\ \ Spin_{21}\stackrel{2y_{20+...}}{\gets} 
Spin_{19}\stackrel{(2,v_1)y_{18}+...}{\gets}
Spin_{17}\stackrel{0}{\gets}
Spin_{15}\stackrel{2y_{14},...}{\gets} \]\[
Spin_{13}\stackrel{2y_{12},(4,2v_1,v_1^2)y_6y_{12},...}{\gets} Spin_{11}.\]
As for restrictions. we see 
\[ c_{10}\mapsto 2y_{20},\quad (c_{9},c_{8}) \mapsto (2,v_1)y_{18},\quad
c_7\mapsto 2y_{14} \quad c_6\to 2y_{12},\quad\]
\[  (c_6c_3,c_6c_2,c_3c_4-v_1e_8)\mapsto (4,2v_1,v_1^2)y_6y_{12}.\]
Thus we can write down some elements in Chow groups
\[(3)\quad Spin_{21}\stackrel{c_{10}}{\gets} 
Spin_{19}\stackrel{c_{9},c_8, e_{16}}{\gets}
Spin_{17}\stackrel{0}{\gets}
Spin_{15}\stackrel{c_{7},...}{\gets}  Spin_{13}\stackrel{c_{6},c_6c_2,c_6c_3,...}{\gets} Spin_{11}.\]
We can see 
$CH^*(R(\bS pin_{21}))/2\supset
\bZ/2\{1,c_{10},c_9,c_8,c_7,...,c_2,c_6c_2,c_6c_3,c_3c_4\}.$ \\
But the above inclusion is far from the isomorphism.

Here we note a counterexample to the Karpenko's conjecture
[Ka2.3].[Ya3].
We consider 
\[ x=c_2c_3c_6c_7\quad \in\ CH^*(R(\bS pin_{17})/2.\]
This element is still defined in  $ CH^*(\bS pin_{15})/2$ but
it is zero there.  We can prove 
\[   x\not =0 \in CH^*(R(\bS pin_{17}))/2\quad but \quad
 x=0\in grK^0(R(\bS pin_{17}))/2.\]
Here $grK^0(X)$ is the graded ring by the gamma filtration of the algebraic $K$-theory $K^0(X)$.

	\section{the field $\bR$ of real numbers}

Let $\bR$ be the field of the real numbers.  
Let $\bG(\bR)$ be a nonzero $G_k$-torsor over $\bR$.  Then
the motivic cohomology for $Spec(\bR)$ is written  
\[  H^{*,*'}(Spec(\bR);\bZ/2)\cong \bZ/2[\tau,\rho],\quad deg(\rho)=(1,1), \ deg(\tau)=(0,1).\]
It is known [Ya1]
\[ CH^*(R_2(\bR))/2\cong \bZ/2\{1,\tau^{-2}\rho^4,\tau^{-3}\rho^6\}\subset \bZ/2[\tau,\rho][\tau^{-1}]. \]
Writing $c_2=\tau^{-2}\rho^4,$  $c_3=\tau^{-3}\rho^6$, we have
\begin{lemma}  ([Ya1]) Let $G=Spin_7$.  Then
   \[ R(\bG(\bR))\cong R(\bG(v))\cong  R_2(\bR)\ (Rost\ motive).\]
  \[ 
 CH^*(R(\bG(\bR))\cong \bZ/2\{c_2\}\oplus \bZ_{(2)}
\{1,c_3\}.\]
\end{lemma}

We note that the map from Chow ring to the etale cohomology  
\[CH^*(R_n(\bR))/2\to H_{et}^{2*}(R_n(\bR);\bZ/2)
\] is injective (Corollary 1.2, Theorem 1.3 in \cite{YaS})
for the Rost motive $R_n$.


For $G=Spin_{11}$, we have $P(y)\cong \Lambda(y_6,y_{10})$.
There is the map $CH^*(R(\bG(v))\to CH^*(R(\bG(\bR))$.
The problem is $y_{10}\in Res(R(\bG(\bR)))$ or not.  Hence we see \\
$ CH^*(R(S pin_{11}(\bR))/2$ is isomorphic to
\[ CH^*(R(\bS pin_{11}))/2
\quad or \quad 
  CH^*(R(\bS pin_{7})/2\otimes \Lambda(y_{10}).\]

We see that  the first case does not happen by the following reason.
A quadratic form $\phi$ (and its quadric) is called excellent 
if for all extension $K$ of $k$, the anisotropic part is 
defined over $k$.  Hence when $k=\bR$, all quadrics are excellent.  It is known that the motive $M$ of the excellent quadric
$X_{\phi}$ is written as a sum of (original) Rost motives [Ro].
(A quadric of trivial discriminant and Cliford invariant for the maximal quadratic grassmannian 
corresponds $Spin$-torsor   [Ka1,2].) 
Therefore we see
\begin{lemma}  
   We have the isomorphism
\[ CH^*(R(S pin_{11}(\bR))/2
\cong  
  CH^*(R(\bS pin_{7})/2\otimes\Lambda(y_{10}).\]
\end{lemma}
\begin{proof} Chow ring of the (original) Rost motive is written by the sum of $py,v_1y,...,v_{n-1}y$ in $Res _{\Omega}$.
From 5.3, we see $0\not =c_2c_4\mapsto
v_1^2y_6y_{10}$.  Hence this element is not in of the
sum of  the original Rost modules.
\end{proof}

\begin{cor}
Let $G=Spin_{2\ell+1}$.  
If we have $c_2c_4\not=0\in CH^*(R(\bG(v)|_L)$
for some extension $L$ over $K$,  then for any nontrivial $\bG(\bR)$, we have  
\[CH^*(R(\bG(v)|_{L})/2 \not \cong CH^*(R(\bG(\bR))/2.\] \end{cor}

{\bf Remark.}  For $G=Spin_{2\ell+1}$, $\ell>5$, it is known
the torsion index $t(G)\ge 4$.  Hence 
$CH^*(R(\bG(v))|_{\bR})\not \cong CH^*(R(\bG(\bR)\otimes \Lambda(y)$.

We prove the similar (and stronger) facts for the etale cohomology
without arguments of excellent quadrics.

\begin{lemma}  Let $G=Spin_{11}$ and $G'=Spin_{7}$.  Then
\[H_{et}^*(R(\bG(\bR);\bZ/2)\cong H_{et}^*(R(\bG'(\bR));\bZ/2)\otimes \Lambda(y_{10}).\]
\end{lemma}
\begin{proof}

Consider the Borel spectral sequence
\[ E_2^{*,*'} \cong H^*(B\bZ/2;\bZ/2)
\otimes H_{et}^{*'}(X|_{\bC};\bZ/2)  \Longrightarrow 
H_{et}^{*}(X;\bZ/2)\]
for $X=R(\bG(\bR))$ or $X=R(\bG'(\bR))$.
  Here 
$ H^*(B\bZ/2;\bZ/2)\cong \bZ/2[\rho],\ |\rho|=1$, and 
\[H_{et}^*(X|_{\bC};\bZ/2)
\cong \Lambda(y_6,y_{10})\quad (or\ \Lambda(y_6) \ for \ G')\]

For $G'$, we see $ d_7(y_6)=\rho^7$ since $H_{et}^*(X;\bZ/2)$ is finite.  Hence 
$E_7^{*,*'}\cong \bZ/2[\rho]/(\rho^7).$
By  dimensional reason $E_7^{*,*'}
\cong E_{\infty}^{*.*'}$.
Thus we have 
\[ H_{et}^*(R(\bG');\bZ/2)\cong E_7^{*.*'}\cong \bZ/2[\rho]/(\rho^7).\]

Similarly, $d_7(y_6)=\rho^7$ also for $G$ by the naturality of
the map $G'\subset G$.  Hence for $G$,  we see
$E_7^{*,*}\cong \bZ/2[\rho]/(\rho^7)\otimes \Lambda(y_{10}).$ 
Moreover by the dimensional reason $E_7^{*.*'}\cong E_{\infty}^{*,*'}$.
Hence we have 
$ H_{et}^*(R(\bG);\bZ/2)\cong \bZ/2[\rho]/(\rho^7)\otimes 
\Lambda(y_{10}).$
\end{proof}

Similarly, we have

\begin{cor}   Let $G$ be a simply connected simple group
having $2$-torsion in $H^*(G)$.
Let $G'=Spin_7$.
Then we have the isomorphism
\[H_{et}^*(R(\bG(\bR);\bZ/2)\cong H_{et}^*(R(\bG'(\bR));\bZ/2)\otimes P(y)/(y_6).\]
\end{cor}



\end{document}